\documentclass[11pt]{amsart}

\usepackage[T1]{fontenc}
\usepackage{lmodern}
\usepackage{microtype}
\usepackage{amsmath,amssymb,amsthm,mathtools}

\usepackage[
colorlinks=true,
linkcolor=red,
citecolor=blue,
urlcolor=blue
]{hyperref}

\usepackage[a4paper,margin=1.15in]{geometry}

\newcommand{\bp}{\bar{\partial}}
\newcommand{\C}{\mathbb C}
\newcommand{\Z}{\mathbb Z}
\newcommand{\OO}{\mathcal O}
\newcommand{\ddbar}{\partial\bar\partial}

\numberwithin{equation}{section}

\newtheorem{theorem} {Theorem} [section]
\newtheorem{bigthm}{Theorem}
 
\newtheorem{proposition}[theorem]{Proposition}

\newtheorem{lemma}  [theorem]     {Lemma}
\theoremstyle{definition}
\newtheorem{example}  [theorem]     {Example}
\newtheorem{definition}  [theorem]     {Definition}
\theoremstyle{plain}
\newtheorem{question}[theorem]{Question}

\theoremstyle{plain}

\theoremstyle{definition}
\newtheorem{remark}  [theorem]     {Remark}

\newtheoremstyle{stepstyle}
{8pt}          
{8pt}          
{\itshape}     
{0pt}          
{\bfseries}    
{.}            
{.5em}         
{}

\theoremstyle{stepstyle}
\newtheorem{step}{Step}[section]

\title[A numerically flat rank-two bundle on a $\ddbar$-threefold]
{A Numerically Flat Rank-Two Bundle without a Holomorphic Connection
	on a $\ddbar$-Threefold}
\let\oldmaketitle\maketitle
\renewcommand\maketitle{{\bfseries\boldmath\oldmaketitle}}	

\author{Tianzhi Hu}
\address{ School of Mathematics and Statistics, Wuhan University, Luojiashan, Wuchang, Wuhan, Hubei, 430072, P.R. China}
\email{hutianzhi@whu.edu.cn}

\author{Runze Zhang}
\address{The Institute of Mathematical Sciences and Department of Mathematics, The Chinese University of
	Hong Kong, Shatin, N.T., Hong Kong}
\email{runzezhang@cuhk.edu.hk, runze.zhang@unice.fr}
\urladdr{\href{https://sites.google.com/view/runzezhang}{https://sites.google.com/view/runzezhang}}

\date{\today}

\begin{document}
	\begin{abstract}
		We construct a numerically flat holomorphic vector bundle of rank two
		on a compact complex threefold satisfying the ordinary $\ddbar$-lemma
		and prove that it admits no holomorphic connection.
		This gives, in particular, a negative answer to a question posed by
		Cao--Deng--Matsumura.
		In contrast, for compact simply connected complex manifolds $X$,
		we prove that the answer is affirmative if
		$H^1(X,\mathcal O_X)=0$, and hence if the Fr\"olicher spectral
		sequence of $X$ degenerates at $E_1$.
	\end{abstract}
	\maketitle
	\section{Introduction}
	Very recently, Cao--Deng--Matsumura have extended a significant
	characterization of numerically flat vector bundles on compact
	K\"ahler manifolds due to Demailly--Peternell--Schneider
	\cite[Theorem 1.18]{DPS} to the \textit{non-K\"ahler setting},
	giving an affirmative answer to the question raised in
	\cite[Remark 1.21]{DPS}.
	
	\begin{theorem}[{\cite[Theorem 1.3]{CDM}}]\label{numflatnonkahler}
		Let $X$ be a compact complex manifold and let $E\to X$ be a
		holomorphic vector bundle.
		Assume that $E$ is numerically flat, namely that $c_1(E)=0$
		and $\mathcal O_{\mathbb P(E)}(1)$ is nef.
		Then $E$ admits a filtration by holomorphic subbundles
		\begin{equation}\label{filtra1}
			\{0\}=E_0\subset E_1\subset\cdots\subset E_k=E
		\end{equation}
		such that each graded quotient $Q_i:=E_i/E_{i-1}$ is Hermitian flat.
		In particular, all Chern classes $c_k(E)$ vanish.
	\end{theorem}
	
	They then propose the following question:
	\begin{question}[{\cite[Question 5.4]{CDM}}]\label{Q_1}
		Let $X$ be an arbitrary compact complex manifold and let
		$E\to X$ be a holomorphic vector bundle.
		If $E$ is numerically flat, does $E$ admit a flat holomorphic
		connection that is furthermore compatible with the
		filtration \eqref{filtra1}?
	\end{question}

	Notice that Question~\ref{Q_1} has a positive answer if $X$ is
	K\"ahler \cite{Simpson,Deng}, or more generally if $X$ lies in
	\textit{Fujiki's class $\mathcal C$}, i.e., if $X$ is bimeromorphic
	to a compact K\"ahler manifold (in which case we simply call $X$
	a \textit{Fujiki manifold}) \cite{Biswas}.
	Given Theorem~\ref{numflatnonkahler}, a unified proof of these
	affirmative answers uses the Hermitian flat-twisted
	$\ddbar$-lemma (or simply the flat-twisted $\ddbar$-lemma),
	which we recall below.
	
	\medskip
	\begin{definition}
		Let $(F,h_F)$ be a Hermitian flat holomorphic vector bundle
		over an arbitrary compact complex manifold $X$, with flat
		Chern connection
		$\nabla_F=\nabla_F^{1,0}+\bar\partial_F$.
		We say that the pair $(X,F)$ satisfies the
		\emph{flat-twisted $\ddbar$-lemma} if the
		$\nabla_F^{1,0}\bar\partial_F$-lemma holds for $(X,F)$.
		That is, for every $\nabla_F$-closed $F$-valued $(p,q)$-form
		$\alpha$ on $X$, the following are equivalent:
		$$
		\nabla_F\text{-exactness}
		\;\Longleftrightarrow\;
		\nabla_F^{1,0}\text{-exactness}
		\;\Longleftrightarrow\;
		\bar\partial_F\text{-exactness}
		\;\Longleftrightarrow\;
		\nabla_F^{1,0}\bar\partial_F\text{-exactness}.
		$$
	\end{definition}

	The validity of the above lemma for \textit{every} Hermitian
	flat holomorphic vector bundle $F$ on a Fujiki manifold $X$
	is well known to experts; nevertheless, we include a proof
	in Proposition~\ref{Fujiki} in Appendix~\ref{appendix}.
	Note that Angella--Kasuya \cite{AKtwisted} studied a related
	$\ddbar$-lemma for twisted differentials associated with
	rank-one Higgs bundles.
	When $F$ has rank one and its connection can be written in
	a global smooth unitary frame as
	$\nabla_F=d+(\theta_2-\bar\theta_2)\wedge$,
	where $\theta_2$ is a global $d$-closed $(1,0)$-form,
	the above definition is precisely the special case
	$\theta_1=0$ of theirs.
	
	\medskip
	We emphasize that the \textit{(ordinary) $\ddbar$-lemma}
	is the special case $F=\mathcal O_X$, equipped with its
	trivial flat connection $d$, and we say that a compact complex
	manifold is a \textit{$\ddbar$-manifold} if it satisfies
	the $\ddbar$-lemma.

	To construct a flat holomorphic connection on $E$
	compatible with the filtration \eqref{filtra1}, however,
	the argument uses coefficients in the Hermitian flat bundles
	$\operatorname{Hom}(Q_j,Q_i)$.
	This distinction is important when considering
	Question~\ref{Q_1} beyond the Fujiki setting.

	\medskip
	We  recall a construction of Biswas--Florentino, 
	which already gives a negative answer to Question \ref{Q_1}
	on a compact complex manifold $X$ that is \textit{not a $\ddbar$-manifold}.
	The numerically flat bundle constructed in their example admits
	no holomorphic connection at all, let alone a flat one compatible
	with the filtration.
	
	\begin{example}[{\cite{BF}}]\label{BF-eg}
		Let $M_{u,v}$ be a Calabi--Eckmann manifold of complex dimension
		$u+v+1$, diffeomorphic to the product of two odd-dimensional spheres
		$S^{2u+1}\times S^{2v+1}$, as constructed by Calabi--Eckmann
		\cite{CE53}, where $u,v\geq0$ and $u+v\geq1$.
		When $u=0$ or $v=0$, it reduces to a Hopf manifold.
		In this case, the Fr\"olicher spectral sequence degenerates at
		the $E_1$-page, but $M_{u,v}$ is not a $\ddbar$-manifold; see \cite[Theorem 3]{Mal91}.
		On the other hand, when $uv>0$, the Fr\"olicher spectral sequence
		does not degenerate at the $E_1$-page; see
		\cite[Remark 4.73]{FOT08}.
		
		\medskip
		\textit{We now assume $u,v\geq1$ and set $X:=M_{u,v}$}, which is then
		simply connected.
		Since $H^1(X,\mathcal O_X)\simeq\mathbb C$ \cite{Hofer},
		following Biswas--Florentino \cite{BF}, we choose a nonzero
		extension class and the corresponding nonsplit extension
		$$
		0\longrightarrow\mathcal O_X
		\longrightarrow E
		\overset{p}{\longrightarrow}\mathcal O_X
		\longrightarrow0.
		$$
		Applying \cite[Proposition 1.15 (ii)]{DPS} to this sequence and
		its dual shows that both $E$ and $E^*$ are nef (see Lemma \ref{lem:extension-nf}).
		Thus $E$ is numerically flat. 
		
		\medskip
		Indeed, \textit{$E$ admits no holomorphic connection}.
		Suppose otherwise, and let $D$ be such a connection.
		The extension and its dual, together with
		$H^0(X,\Omega_X^2)=0$ \cite{Hofer}, give
		$
		H^0\bigl(X,\Omega_X^2\otimes\operatorname{End}(E)\bigr)=0.
		$
		Hence $D$ is flat and, since $X$ is simply connected,
		$E\simeq\mathcal O_X^{\oplus2}$.
		Under this trivialization, $p$ is a surjection
		$\mathcal O_X^{\oplus2}\to\mathcal O_X$ given by constants,
		as $X$ is compact and connected.
		It therefore has a constant right inverse, contradicting
		the nonsplitting of the extension.
	\end{example}
	
	In Example~\ref{BF-eg}, the relevant coefficient bundle is
	$\operatorname{Hom}(\mathcal O_X,\mathcal O_X)\simeq\mathcal O_X$,
	so the flat-twisted $\ddbar$-lemma reduces to the ordinary one,
	which already fails on $X$.

	\textit{It is therefore natural to ask whether the ordinary
		$\ddbar$-lemma alone suffices to give an affirmative answer
		to Question~\ref{Q_1}.}
	In complex dimension two, the answer is affirmative:
	every compact $\ddbar$-surface has even first Betti number
	and is therefore K\"ahler by e.g. \cite[Theorem~11]{Buchdahl}.
	Our main result shows that this is no longer true in general
	in higher dimensions, already in dimension three.
	
	\begin{bigthm}\label{thm:main}
		There exist a $\ddbar$-threefold $X$  and a numerically flat holomorphic bundle $E\rightarrow X$
		of rank two which admits no holomorphic connection.
	\end{bigthm}

	More explicitly, we choose a compact $\partial\bar\partial$-threefold
	$X$ among the manifolds studied by Rubini \cite[\S2.3]{Rubini}.
	It admits a holomorphic fibration $p:X\to B$ over an elliptic curve
	$B$, with two-dimensional complex tori as fibers.
	We construct an extension
	$$0\longrightarrow L\longrightarrow E \longrightarrow \mathcal O_X\longrightarrow  0$$ with $L$ Hermitian flat,
	so that $E$ is numerically flat.
	
	Locally over $B$, after choosing a flat holomorphic frame of $L$,
	the extension is defined by
	$$
	\eta=f_1(z_3)d\bar z_1+f_2(z_3)d\bar z_2,
	$$
	where $z_3$ is a local coordinate on $B$ and $z_1,z_2$ are
	local coordinates along the torus fibers.
	We take $f_1=e^{-2iz_3}$ and $f_2=e^{2iz_3}$, so that
	$$
	f_1f_2'-f_2f_1'=4i.
	$$
	Integration over a torus fiber shows that a holomorphic connection
	would force this expression to vanish.
	Hence $E$ admits no holomorphic connection.
	
	\medskip
	\medskip
	Inspired by Example \ref{BF-eg}, we prove
	\begin{bigthm}
		\label{thm:simply-connected}
		Question \ref{Q_1} has a positive answer for any compact
		simply connected complex manifold $X$ with
		$H^1(X,\mathcal O_X)=0$.
		The vanishing condition holds whenever the Fr\"olicher spectral
		sequence of $X$ degenerates at $E_1$, in particular when $X$
		is a $\ddbar$-manifold.
	\end{bigthm} 
	
	Examples illustrating this result beyond Fujiki's class
	$\mathcal C$ are given in \ref{cle} and \ref{twistor}.
	
	\vspace{.5cm}
	\noindent\textbf{Acknowledgements:}
	The authors are grateful to Professors Hisashi Kasuya and Shin-ichi Matsumura  for many valuable comments. 
	
	\vspace{.5cm}
	\noindent\textbf{AI Declaration:} 
	ChatGPT-6 brought the example of Biswas--Florentino \cite{BF}
	to the authors' attention.
	Motivated by this example, the authors raised the question of
	whether numerical flatness implies the existence of a flat
	holomorphic connection on a $\ddbar$-manifold and developed
	the main ideas of this paper.
	Artificial intelligence tools were also used for calculations,
	verification of intermediate results, and polishing of the writing.
	All mathematical statements have been independently checked
	by the authors, who take full responsibility for the content.
	
	\section{Proof of Theorems \ref{thm:main} and \ref{thm:simply-connected}}
	
	\begin{proof}[{Proof of Theorem \ref{thm:main}}]
		The proof  is constructive and proceeds in the following three steps:
		
		\begin{step}
			Construction of the $\ddbar$-threefold.
		\end{step}
		We choose $X$ to be a particular member of the family of compact
		complex threefolds associated with the six-dimensional real solvable
		Lie algebra $\mathfrak g_8$, studied by Rubini \cite[\S2.3]{Rubini}.
		For earlier work on manifolds in this family, see
		\cite{Kasuya,Otal,FOU,AK}.  Following \cite[equation (11), p. 8]{Rubini}, take
		coordinates $(z_1,z_2)$ on $\C^2$ and $z_3$ on $\C$, and set
		$$
		G=\C\ltimes_{\varphi_A}\C^2,\quad
		\varphi_A(z_3)=\operatorname{diag}(\alpha_1(z_3),\alpha_2(z_3)),
		$$
		$$
		\alpha_1(z_3)=e^{-(A-i)z_3-(A+i)\bar z_3},\quad
		\alpha_2=\alpha_1^{-1}.
		$$
		Here $\ltimes$ denotes a \emph{semidirect product}: the underlying space is
		$\C\times\C^2$, and multiplication is
		$$
		(a,w)\cdot(z_3,z)=(a+z_3,w+\varphi_A(a)z).
		$$
		With the usual complex coordinates, every left translation is holomorphic.
		
		Take
		\begin{equation*}\label{eq:parameters}
			A=-2i,\quad n=3,\quad n'=-1,\quad
			\lambda=\frac{3+\sqrt5}{2},\quad \zeta=e^{2\pi i/3}.
		\end{equation*}
		The corresponding base lattice in \cite[Table 1, p. 8]{Rubini} is
		\begin{equation}\label{eq:base-lattice}
			\Gamma'=u\Z\oplus v\Z,\quad
			u=-\frac\pi6,\quad v=\frac{i\log\lambda}{2}.
		\end{equation}
		Choose the fiber lattice
		\begin{equation*}\label{eq:fiber-lattice}
			\begin{aligned}
				\Gamma''=\operatorname{Span}_{\Z}\bigl\{(1,1),\ (\lambda,\lambda^{-1}),\ (\zeta,\zeta^{-1}),\ (\lambda\zeta,\lambda^{-1}\zeta^{-1})\bigr\}
				\subset\C^2.
			\end{aligned}
		\end{equation*}
		The four generators are linearly independent over $\mathbb R$, so
		$$
		T:=\C^2/\Gamma''
		$$
		is a compact complex torus. The identities
		$\lambda^2-3\lambda+1=0$ and $\zeta^2+\zeta+1=0$ show that
		$$
		P=\operatorname{diag}(\lambda,\lambda^{-1}),\quad
		Q=\operatorname{diag}(\zeta,\zeta^{-1})
		$$
		preserve $\Gamma''$, as do their inverses $3I-P$ and $-I-Q$.
		Since
		$$
		\varphi_A(u)=Q^{-1},\quad \varphi_A(v)=P^{-1},
		$$
		both generators of $\Gamma'$ act by automorphisms of the fiber lattice
		$\Gamma''$. Thus
		$$
		\Gamma:=\Gamma'\ltimes_{\varphi_A}\Gamma''
		$$
		is a lattice in $G$. Set
		$$
		X:=\Gamma\backslash G.
		$$
		Then $X$ is a compact complex threefold. Explicitly, it is the quotient
		of $\C_{z_3}\times T$ by
		\begin{equation}\label{eq:deck}
			\begin{aligned}
				g_u(z_3,z_1,z_2)&=(z_3+u,\zeta^{-1}z_1,\zeta z_2),\\
				g_v(z_3,z_1,z_2)&=(z_3+v,\lambda^{-1}z_1,\lambda z_2).
			\end{aligned}
		\end{equation}
		For these parameters, $X$ belongs to 
		\cite[case (i) of Theorem 2.5, p. 9]{Rubini}. By Kasuya's result
		\cite[Corollary 4.2]{Kasuya} and the calculation in
		\cite[Table 7, p. 19]{Rubini}, its Dolbeault cohomology is computed
		by the span of the wedge products of
		$$
		dz_3,\quad d\bar z_3,\quad dz_1\wedge dz_2,\quad
		d\bar z_1\wedge d\bar z_2.
		$$
		The same span computes Bott--Chern cohomology by
		\cite[Theorem 2.16]{AK}.
		Since it consists of $d$-closed forms and is preserved by complex
		conjugation, \cite[Lemma 4.2.13]{Otal} gives the $\ddbar$-lemma for $X$.
		
		Notice also that $X$ however does not belong to Fujiki's class
		$\mathcal C$.
		Indeed, suppose otherwise. Since $X$ is a complex solvmanifold,
		Arapura's theorem \cite[Theorem~9]{Arapura}
		(see also \cite[Theorem~3.3]{AKtwisted})
		implies that $X$ admits a K\"ahler form $\omega$.
		However, the above computation gives
		$H^{1,1}_{\mathrm{BC}}(X)=\mathbb C[i\,dz_3\wedge d\bar z_3]$.
		Since this generator has square zero, we would have
		$0<\int_X\omega^3=0$, a contradiction.
		
		\begin{step}
			The numerically flat extension.
		\end{step}
		Let $$q:\C_{z_3}\times T\to X$$ be the quotient covering.
		By \eqref{eq:deck} and the $\mathbb R$-linear independence of $u,v$
		in \eqref{eq:base-lattice}, the deck transformation group of $q$ is
		$\langle g_u,g_v\rangle\simeq\Z^2$. Define the unitary character
		\begin{equation*}\label{eq:character}
			\chi:\langle g_u,g_v\rangle\longrightarrow\{\pm1\},
			\quad \chi(g_u^m g_v^n)=(-1)^m.
		\end{equation*}
		Let $\pi:L\to X$ be the holomorphic line bundle obtained from
		$(\C_{z_3}\times T)\times\C$ by the identifications
		$$
		(x,s)\sim(gx,\chi(g)s),
		\quad g\in\langle g_u,g_v\rangle,
		$$
		where $x\in\C_{z_3}\times T$ and $s\in\C$, with 
		$\pi([x,s])=q(x)$.
		
		Since $|\chi(g)|=1$, the metric $|s|^2$ descends to a Hermitian
		metric $h$ on $L$. A local holomorphic lift
		$\sigma_U:U\to\C_{z_3}\times T$ of $q$ gives a holomorphic frame
		$e_U(y)=[\sigma_U(y),1]$ with $h(e_U,e_U)=1$. Hence its Chern curvature
		$$
		F_h|_U=-\frac{i}{2\pi}\partial\bar\partial\log h(e_U,e_U)=0,
		$$
		so $L$ is Hermitian flat.
		We then recall the following standard result.
		
		\begin{lemma}\label{lem:extension-nf}
			Let $L$ be a  Hermitian flat line bundle. Then every holomorphic extension
			$$
			0\longrightarrow L\longrightarrow E\longrightarrow\OO_X\longrightarrow0
			$$
			is numerically flat.
		\end{lemma}
		
		\begin{proof}
			Since $L$ and $L^*$ are Hermitian flat, they are nef. 
			\cite[Proposition 1.15 (ii)]{DPS} (which holds on any compact complex
			manifold) can then be applied to obtain that $E$ is nef. Applying the same result to the dual exact sequence
			\[
			0\longrightarrow\OO_X\longrightarrow E^*
			\longrightarrow L^*\longrightarrow0
			\]
			shows that $E^*$ is also nef. Hence $E$ is numerically flat.
		\end{proof}
		
		We now construct the extension. On $\C_{z_3}\times T$, set
		\begin{equation}\label{eq:eta}
			\eta=e^{-2iz_3}d\bar z_1+e^{2iz_3}d\bar z_2.
		\end{equation}
		By \eqref{eq:deck}, we have
		$$
		g_u^*\eta=-\eta,\quad g_v^*\eta=\eta.
		$$
		Hence $\eta$ descends to an $L$-valued $(0,1)$-form over $X$.
		Since its coefficients are holomorphic in $z_3$, we have $\bp_L\eta=0$.
		The class
		$$
		\xi:=[\eta]\in H^{0,1}(X,L)
		\simeq\operatorname{Ext}^1_{\OO_X}(\OO_X,L)
		$$
		defines an extension
		\begin{equation*}\label{eq:bundle-extension}
			0\longrightarrow L\longrightarrow E_\xi\longrightarrow\OO_X\longrightarrow0.
		\end{equation*}
		Explicitly,
		\begin{equation}\label{eq:bundle}
			E_\xi:=\bigl(L\oplus\OO_X,\bp_{E_\xi}\bigr),\quad
			\bp_{E_\xi}=\begin{pmatrix}\bp_L&\eta\\0&\bp\end{pmatrix},
		\end{equation}
		where $L\oplus\OO_X$ denotes the underlying smooth bundle.
		The equality $\bp_L\eta=0$ gives $\bp_{E_\xi}^2=0$.
		By Lemma \ref{lem:extension-nf}, $E_\xi$ is numerically flat.
		
		\begin{step}
			Non-existence of a holomorphic connection.
		\end{step}
		It remains to prove that $E_\xi$ admits no holomorphic connection.
		For this purpose, we first prove the following lemma.
		
		\begin{lemma}\label{lem:obstruction}
			Let $\Delta\subset\C$ be a disc with coordinate $z_3$, and let
			$T=\C^2/\Lambda$ be a complex torus. For $f_1,f_2\in\OO(\Delta)$, set
			$$
			\eta=f_1(z_3)d\bar z_1+f_2(z_3)d\bar z_2.
			$$
			Let $F$ be the holomorphic rank-two bundle on $\Delta\times T$ whose underlying
			smooth bundle is trivial and whose Dolbeault operator is
			$$
			\bp_F=\begin{pmatrix}\bp&\eta\\0&\bp\end{pmatrix}.
			$$
			If $F$ admits a holomorphic connection, then
			\begin{equation*}\label{eq:obstruction}
				f_1f_2'-f_2f_1'=0\quad\hbox{on }\Delta,
			\end{equation*}
			where primes denote differentiation with respect to $z_3$.
		\end{lemma}
		
		\begin{proof}
			Suppose that $F$ admits a holomorphic connection $D$. Let $e_1,e_2$
			be the standard smooth frame of $F$. Put
			$D_3:=D_{\partial/\partial z_3}$ and write
			$$
			D_3e_1=a_{11}e_1+a_{21}e_2,\quad
			D_3e_2=a_{12}e_1+a_{22}e_2,
			$$
			where $a_{kl}\in C^\infty(\Delta\times T)$. Denote by $\bp_j$ the component
			of $\bp_F$ in the $\partial/\partial\bar z_j$ direction. By definition,
			$$
			\bp_je_1=0,\quad \bp_je_2=f_je_1,\quad j=1,2.
			$$
			Since $D$ is holomorphic, we have
			$\bp_jD_3e_2=D_3\bp_je_2$. Comparing the coefficients of
			$e_1$, we obtain
			\begin{equation}\label{eq:two-equations}
				\frac{\partial a_{12}}{\partial\bar z_j}
				+(a_{22}-a_{11})f_j=f_j',\quad j=1,2.
			\end{equation}
			Let $d\mu$ be the translation-invariant volume form on $T$ with
			$\int_Td\mu=1$. Since $f_1,f_2$ depend only on $z_3$, it follows from
			\eqref{eq:two-equations} that
			$$
			\begin{aligned}
				f_1f_2'-f_2f_1'
				&=\int_T\left(f_1\frac{\partial a_{12}}{\partial\bar z_2}
				-f_2\frac{\partial a_{12}}{\partial\bar z_1}\right)d\mu\\
				&=0,
			\end{aligned}
			$$
			where the last equality follows by integration by parts on $T$.
		\end{proof}
		
		
		We now apply Lemma \ref{lem:obstruction} to $E_\xi$. 
		
		Projection onto the first factor gives
		$$
		p:X\rightarrow B:=\C/\Gamma'.
		$$
		Choose a small disc $\Delta\subset B$ and lift it to the $z_3$-plane. Then
		$p^{-1}(\Delta)\simeq\Delta\times T$, and
		$L|_{p^{-1}(\Delta)}$ has a flat holomorphic trivialization.
		By \eqref{eq:eta} and \eqref{eq:bundle}, $E_\xi|_{p^{-1}(\Delta)}$ is the bundle in
		Lemma \ref{lem:obstruction}, with
		$$
		f_1(z_3)=e^{-2iz_3},\quad f_2(z_3)=e^{2iz_3}.
		$$
		We compute
		$$
		f_1f_2'-f_2f_1'=4i\neq0,
		$$
		so Lemma \ref{lem:obstruction} shows that
		$E_\xi|_{p^{-1}(\Delta)}$ admits no holomorphic connection.
		Consequently, neither does $E_\xi$. This completes the proof of
		Theorem \ref{thm:main}.
	\end{proof}

	\begin{proof}[{Proof of Theorem \ref{thm:simply-connected}}]
		Let $E$ be a numerically flat holomorphic vector bundle on $X$.
		By Theorem~\ref{numflatnonkahler}, there exists a filtration
		$
		\{0\}=E_0\subset E_1\subset\cdots\subset E_k=E
		$
		by holomorphic subbundles such that each quotient
		$Q_i:=E_i/E_{i-1}$ is Hermitian flat.
		Since $X$ is simply connected, parallel transport for the flat
		Chern connection on $Q_i$ gives a global holomorphic frame.
		Hence $Q_i\simeq\mathcal O_X^{\oplus r_i}$, where
		$r_i:=\operatorname{rank}Q_i.
		$
		
		We show inductively that the filtration splits holomorphically.
		Suppose that
		$E_{i-1}\simeq\mathcal O_X^{\oplus s_i}$.
		The extension
		$$
		0\longrightarrow E_{i-1}\longrightarrow E_i
		\longrightarrow Q_i\longrightarrow0
		$$
		splits because
		$$
		\operatorname{Ext}^1_{\mathcal O_X}(Q_i,E_{i-1})
		\simeq
		H^1\bigl(X,\operatorname{Hom}(Q_i,E_{i-1})\bigr)
		\simeq
		H^1(X,\mathcal O_X)^{\oplus r_i s_i}
		=0.
		$$
		Choosing these splittings inductively identifies
		$$
		E\simeq\bigoplus_{i=1}^k Q_i,
		\quad
		E_j\simeq\bigoplus_{i=1}^j Q_i.
		$$
		The direct sum of the trivial connections on the $Q_i$ is
		therefore a flat holomorphic connection preserving the filtration.

		Finally, if the Fr\"olicher spectral sequence of $X$ degenerates
		at $E_1$, then simply connectedness gives
		$$
		0=b_1(X)=h^{1,0}(X)+h^{0,1}(X).
		$$
		Hence $H^1(X,\mathcal O_X)=0$. The proof of Theorem \ref{thm:simply-connected} is completed.
	\end{proof}

	The following examples illustrate the scope of
	Theorem~\ref{thm:simply-connected} beyond Fujiki's class
	$\mathcal C$.
	The first satisfies the $\ddbar$-lemma, while the second
	has $E_1$-degeneration but does not satisfy the $\ddbar$-lemma.
	
	\begin{example}[Clemens manifolds]\label{cle}
		We recall Clemens' construction as described in
		\cite[p.~1001]{Friedman}.
		Let $Y$ be a projective Calabi--Yau threefold containing
		disjoint smooth rational curves $C_1,\ldots,C_r$ with
		$$
		N_{C_i/Y}\simeq
		\mathcal O_{\mathbb P^1}(-1)\oplus\mathcal O_{\mathbb P^1}(-1).
		$$
		Assume that their classes span $H^4(Y,\mathbb C)$ and satisfy
		a relation
		$$
		\sum_{i=1}^r m_i[C_i]=0,
		\quad m_i\neq0\quad\text{for every }i.
		$$
		Contracting these curves gives a singular threefold
		$\overline Y$ with ordinary double points.
		Its small smoothings are compact complex threefolds $X$
		with $b_2(X)=0$, called \emph{Clemens manifolds}.
		
		One can choose $Y$ to be simply connected and the classes
		$[C_i]$ to generate $H^4(Y,\mathbb Z)$.
		Then $X$ is diffeomorphic to a connected sum of copies of
		$S^3\times S^3$, and hence is simply connected;
		see \cite[p.~1001]{Friedman}.
		Friedman proved that general Clemens manifolds satisfy the
		$\ddbar$-lemma\footnote{Here ``general'' means that the smoothing
			parameter lies outside a proper real analytic subset.}
		\cite[Corollary~3.8 \& Theorem~3.10]{Friedman}.
		Li subsequently proved that every sufficiently small
		smoothing is a $\ddbar$-manifold \cite{Li}.
		
		These manifolds do not belong to Fujiki's class $\mathcal C$.
		Indeed, by \cite[Lemma 5.15 \& Remark 5.16]{DGMS75}, the $\ddbar$-lemma and $b_2(X)=0$ give
		$$
		H^{1,1}_\textrm{BC}(X)\hookrightarrow
		H^2_{\mathrm{dR}}(X,\mathbb C)=0.
		$$
		On the other hand, a manifold in Fujiki's class $\mathcal C$
		admits a K\"ahler current \cite[Theorem~0.7]{DP},
		whose Bott--Chern class is nonzero.
	\end{example}
	
	\begin{example}[The twistor space of a K3 surface]\label{twistor}
			Let $S$ be a K3 surface equipped with a hyperk\"ahler metric,
			and let $Z$ be the total space of its twistor family over
			$\mathbb P^1$ \cite[\S3 (F)]{HKLR}.
			The compact complex threefold $Z$ is diffeomorphic to
			$S\times S^2$, so it is simply connected.
			Its Fr\"olicher spectral sequence degenerates at $E_1$
			by \cite[Corollary~5.2]{ES}.
			
			However, $Z$ is not a $\ddbar$-manifold.
			Indeed, the cohomology computations of Eastwood--Singer
			\cite{ES} give
			$$
			h^{2,0}(Z)=0,\quad h^{0,2}(Z)=3,
			$$
			which contradicts the Hodge symmetry required by the
			$\ddbar$-lemma.
	\end{example}

	\appendix
\section{Flat twisted-$\ddbar$-lemma on Fujiki manifolds}\label{appendix}
	\begin{proposition}\label{Fujiki}
		Let $X$ be a  Fujiki manifold and let $(F,h_F)$ be any Hermitian flat  holomorphic vector bundle with flat Chern connection $\nabla_F = \nabla_F^{1,0} + \bar\partial_F$. Then the $\nabla_F^{1,0}\bar\partial_F$-lemma holds for $(X,F)$, i.e., for every $\nabla_F$-closed $F$-valued $(p,q)$-form $\alpha$ over $X$, the following are equivalent:
		\[
		\nabla_F\text{-exactness} \;\Longleftrightarrow\; \nabla_F^{1,0}\text{-exactness} \;\Longleftrightarrow\; \bar\partial_F\text{-exactness} \;\Longleftrightarrow\; \nabla_F^{1,0}\bar\partial_F\text{-exactness}.
		\]
	\end{proposition}
	
	\begin{proof}
		If $X$ is K\"ahler, Hermitian flatness  implies, via the Bochner--Kodaira--Nakano identity, that the Laplacians satisfy
		\[
		\square_{\bar\partial_F} = \square_{\nabla^{1,0}_F} = \frac{1}{2}\square_{\nabla_F}.
		\] Hence the classical $\partial\bar\partial$-lemma argument applies to $F$-valued forms, yielding the $\nabla_F^{1,0}\bar\partial_F$-lemma.
		
		\medskip
		Now suppose that $X$ is in Fujiki class $\mathcal C$. We follow the proof of  \cite[$\S$VI, Theorem 12.9]{Deme}. Choose a proper modification $\mu:\widetilde X\to X$ with $\tilde X$ K\"ahler and set $\widetilde F=\mu^*F$. Notice that  this bundle is again Hermitian flat. 
		
		For any smooth $F$-valued form $\alpha$ over $X$, we have $\mu_*\mu^*\alpha = \alpha$ in the sense of  currents, because $\mu$ is a biholomorphism outside a set of measure zero. Consequently, on cohomology $\mu_*$ is surjective and $\mu^*$ is injective, though not necessarily isomorphisms. 
		
		We now have the commutative diagrams
		\[
		\renewcommand{\arraystretch}{1.5}
		\begin{array}{ccc}
			H_{\mathrm{BC}}^{p,q}(\widetilde X,\widetilde F) & \xrightarrow{\;\;\,\,\cong\,\,\;\;} & H^{p,q}_{\bp}(\widetilde X,\widetilde F) \\[1ex]
			{\scriptstyle\mu_*}\Big\downarrow\;\Big\uparrow{\scriptstyle\mu^*} & & {\scriptstyle\mu_*}\Big\downarrow\;\Big\uparrow{\scriptstyle\mu^*} \\[1ex]
			H_{\mathrm{BC}}^{p,q}(X,F) & \xrightarrow{\;\;\quad\,\;\;} & H^{p,q}_{\bp}(X,F)
		\end{array}
		\quad
		\begin{array}{ccc}
			\bigoplus_{p+q=k} H_{\mathrm{BC}}^{p,q}(\widetilde X,\widetilde F) & \xrightarrow{\;\;\,\,\cong\,\,\;\;}  & H_{\mathrm{dR}}^k(\widetilde X,\widetilde F) \\[1ex]
			{\scriptstyle\mu_*}\Big\downarrow\;\Big\uparrow{\scriptstyle\mu^*} & & {\scriptstyle\mu_*}\Big\downarrow\;\Big\uparrow{\scriptstyle\mu^*} \\[1ex]
			\bigoplus_{p+q=k} H_{\mathrm{BC}}^{p,q}(X,F)& \xrightarrow{\;\;\quad\,\;\;} & H_{\mathrm{dR}}^k(X,F),
		\end{array}
		\]	where, for instance, the \textit{flat-twisted Bott--Chern cohomology group} is defined by
		\[
		H_{\mathrm{BC}}^{\bullet,\bullet}(X,F) = \frac{\ker \nabla^{1,0}_F \cap \ker\bar\partial_F}{\operatorname{im} \nabla_F^{1,0}\bar\partial_F}.
		\]
		Notice that the vertical arrows are either both upward or both
		downward. Since the top horizontal arrows are isomorphisms by
		the same argument as in the untwisted case
		(see \cite[Lemma 5.15 \&  Remark 5.16]{DGMS75}),
		the injectivity/surjectivity of the vertical arrows forces the
		bottom horizontal arrows to be isomorphisms as well.   Therefore
		$(X,F)$ satisfies the $\nabla_F^{1,0}\bar\partial_F$-lemma by using \cite{DGMS75} again.
	\end{proof}

	\begin{remark}
		In a forthcoming paper \cite{Zha}, a Tian--Todorov approach
		using Proposition \ref{Fujiki} will be developed to prove that the joint
		deformations of a pair $(X,E)$ are unobstructed when $X$ is a compact
		complex manifold in Fujiki's class $\mathcal C$ with torsion canonical
		bundle and $E\to X$ is a holomorphic vector bundle satisfying
		$H^2\bigl(X,\operatorname{End}^0(E)\bigr)=0$, where
		$\operatorname{End}^0(E)$ denotes the trace-free endomorphism subbundle
		of $\operatorname{End}(E)$.
	\end{remark}

\end{document}